\documentclass[a4paper,11pt]{article}
\usepackage[a4paper]{geometry}
\usepackage{graphicx}
\usepackage{amssymb,amsmath, amsthm}
\usepackage{mathrsfs}
\usepackage{bbm,bbold}

\usepackage[finalnew]{trackchanges}

\newcommand{\myboldmath}[1]{\mathbbm{#1}}
\newcommand{\R}{\myboldmath{R}}

\newcommand{\N}{\myboldmath{N}}

\newcommand{\1}{\mathbf{1}}

\newcommand{\E}{\myboldmath{E}}

\theoremstyle{plain}
\newtheorem{theorem}{Theorem}
\newtheorem{proposition}[theorem]{Proposition}
\newtheorem{lemma}[theorem]{Lemma}
\newtheorem*{lemma*}{Auxiliary Lemma}

\newtheorem{corollary}[theorem]{Corollary}

\theoremstyle{definition}

\theoremstyle{remark}

\DeclareMathOperator{\lk}{lk}

\def\polhk#1{\setbox0=\hbox{#1}{\ooalign{\hidewidth
    \lower1.5ex\hbox{`}\hidewidth\crcr\unhbox0}}}

\begin{document}

\title{On Topological Minors in Random Simplicial Complexes\thanks{Research supported by the Swiss National Science Foundation (SNF Projects 200021-125309 and 200020-138230)}}
\author{Anna Gundert\footnote{Universit\"{a}t zu K\"{o}ln, Weyertal 86--90, 50923 K\"{o}ln, Germany. \texttt{anna.gundert@uni-koeln.de}. Work on this paper was conducted at the Institute of Theoretical Computer Science, ETH Z\"{u}rich.} \and Uli Wagner\footnote{IST Austria, Am Campus 1, 3400 Klosterneuburg, Austria. \texttt{uli@ist.ac.at}.}}

\maketitle

For random graphs, the \emph{containment problem} considers the probability that a binomial random graph $G(n,p)$ contains a given graph as a substructure.
When asking for the graph as a \emph{topological minor}, i.e. for a copy of a \emph{subdivision} of the given graph, it is well-known that the (sharp) threshold is at $p=1/n$.
We consider a natural analogue of this question for higher-dimensional random complexes $X^k(n,p)$,
first studied by Cohen, Costa, Farber and Kappeler for $k=2$.

Improving previous results, we show that $p=\Theta(1/\sqrt{n})$ is the (coarse) threshold for containing a subdivision of any fixed complete $2$-complex. \change{The upper bound for the threshold extends to the case of $k>2$, for $p = \Theta(n^{-1/k})$.}{\protect
For higher dimensions $k>2$, we get that $p=O(n^{-1/k})$ is an upper bound for the threshold probability of containing a subdivision 
of a fixed $k$-dimensional complex.}


\section{Introduction}

A basic problem in graph theory is to determine whether a given graph $G$, which may be thought of as ``large'', contains a fixed graph $H$
as a substructure.
The most straightforward form of containment is that $G$ contains a copy of $H$ as a \emph{subgraph}. 
Another important variant is that $G$ contains some \emph{subdivision} of $H$ as a subgraph; 
in this case, one also says that $G$ contains $H$ as a \emph{topological minor}.

For random graphs, the \emph{containment problem} considers the probability that a binomial random graph $G(n,p)$ contains a copy of a given graph $H$.
For subgraph containment, it is well-known \cite{Bollobas:1981} that this probability has a (coarse) threshold of $\Theta(n^{-1/m(H)})$, where $m(H)$ is the density of the densest subgraph of $H$. The (sharp) threshold for containment of any complete graph of fixed size as a topological minor is  $p=1/n$ by a well-known result of Ajtai, Koml\'{o}s and Szemer\'{e}di\change{\protect\cite{AjtaiKomlosSzemeredi:1979,ErdosRenyi:1960}}{\protect\footnote{The lower bound on the threshold follows from \cite[Theorem 5e]{ErdosRenyi:1960}.} \cite{AjtaiKomlosSzemeredi:1979}.} 
%
%

\add{\protect Here, a sequence $\hat{p}=\hat{p}(n)$ is called a \emph{threshold} for an increasing graph property $\mathcal{Q}$ if the probability $\Pr[G(n,p) \text{ has } \mathcal{Q}]$ tends to $0$ if $p=o(\hat{p})$ and to $1$ if $\hat{p}=o(p)$. A threshold $\hat{p}$ is \emph{sharp} if for any $\epsilon >0$
\[
\Pr[G(n,p) \text{ has } \mathcal{Q}] \xrightarrow[n\rightarrow \infty]{} \begin{cases}0&\text{if }p\leq(1-\epsilon)\hat{p},\\1&\text{if }p\geq(1+\epsilon)\hat{p}.\end{cases} 
\]
Any monotone graph property has a threshold \cite{Bollobas:1987}. If a monotone graph property does not have a sharp threshold, it has a \emph{coarse threshold}. (For more detailed results on thresholds of graph properties, see, e.g., \cite{Friedgut:1999,Friedgut:1996} and \cite[Chapter 1]{Janson:2000} for an overview.)
}

A random model $X^k(n,p)$ for simplicial complexes of arbitrary fixed dimension $k$ which generalizes the random graph model $G(n,p)$ was introduced by Linial and Meshulam~\cite{LinialMeshulam:HomologicalConnectivityRandom2Complexes-2006} and has since then been studied extensively, see, e.g., \cite{Aronshtam:CollapsibilityVanishingTopHomologyRandomComplexes-2013,BabsonHoffmanKahle:SimpleConnectivityRandom2Complexes,CohenCostaFarberKappeler:2012,Kozlov:2009p2037,MeshulamWallach:HomologicalConnectivityRandomComplexes-2009}.
Subgraph containment admits a direct generalization to higher dimensions: We can ask whether a given simplicial complex $X$ contains a fixed complex $K$ as a subcomplex.
The proof methods for random graphs extend directly to random complexes of higher dimension, and the threshold probability for $X^k(n,p)$ to contain a fixed complex $K$ as a subcomplex is given by the density (in terms of \add{\protect number of} $k$-faces versus \add{\protect number of} vertices) of the densest subcomplex of $K$, see~\cite{BabsonHoffmanKahle:SimpleConnectivityRandom2Complexes,CohenCostaFarberKappeler:2012}.

As a natural higher-dimensional analogue of topological graph minors, we say that a complex $X$ has a fixed complex $K$ as a \emph{topological minor} if $X$ contains some subdivision of $K$.
Cohen, Costa, Farber and Kappeler~\cite{CohenCostaFarberKappeler:2012} show that for any $\epsilon >0 $ and $p\geq n^{-1/2 + \epsilon}$, the random complex $X^2(n,p)$ \emph{asymptotically almost surely (a.a.s.)}, i.e., with probability tending to $1$ as $n\rightarrow\infty$, contains any fixed $K$ as a topological minor.
Their method also extends to random complexes $X^k(n,p)$ of higher dimension $k>2$ with $p\geq n^{-1/k + \epsilon}$.

\change{We improve this upper bound on the threshold probability for $2$-dimensional subdivision containment to $p=O(1/\sqrt{n})$ and show an upper bound of $O(n^{-1/k})$ for random $k$-complexes $X^k(n,p)$ with $k > 2$:}{\protect We improve this upper bound on the threshold probability for containing a fixed $k$-dimensional topological minor to $O(n^{-1/k})$:}

\begin{theorem}\label{THM - 1statement}
For every $k \geq 2$ and $t \geq k+1$ there is a constant $c=c(t,k)>1$ such that $X^k(n,p)$ with $p\geq \sqrt[k]{c/n}$ a.a.s.~contains the complete $k$-complex $K_t^k$ on $t$ vertices as a topological minor.
\end{theorem}

\add{\protect Here, the complete $k$-complex $K_t^k$ is the $k$-skeleton of the $(t-1)$-simplex and consists of all sets of cardinality $\leq k+1$ on $t$ vertices. Thus, every finite $k$-complex is contained in $K_t^k$ for large enough $t$.}

The result in \cite{CohenCostaFarberKappeler:2012} is proven by reduction to the subcomplex containment problem, by showing that a given $2$-complex $K$ can be subdivided to decrease its triangle density. For Theorem~\ref{THM - 1statement} we use a different approach, based on an idea going back at least to Brown, Erd\H{o}s and S\'os \cite{Brown:1973ve} and used also in \cite{BabsonHoffmanKahle:SimpleConnectivityRandom2Complexes}:
For $2$-complexes our proof is based on studying common links of pairs of vertices, which form random graphs of the type \mbox{$G(n-2,p^2)$} and then uses results on the phase transition in random graphs. This approach also extends to complexes of dimension $k>2$.
For a (possibly more approachable) sketch of the proof for the case $k=2$ we refer the reader to the extended abstract~\cite{GundertWagner:2013} of this paper.

For $2$-complexes we also \change{show the corresponding $0$-statement}{\protect show a corresponding lower bound} and thus establish that the \remove{(coarse)} threshold for containing a subdivision of any fixed complete $2$-complex is at \remove{$p=$}$\Theta(1/\sqrt{n})$:

\begin{theorem}\label{THM - 0statement}
There is a constant $c<1$ such that for any $t \geq 10$ the random $2$-complex $X^2(n,p)$ with $p = \sqrt{c/n}$ a.a.s.~does not have $K_t^2$ as a topological minor.
\end{theorem}

The somewhat technical proof of Theorem~\ref{THM - 0statement} is based on bounds on the number of triangulations of a fixed surface.

\add{\protect
Before we give the proofs of these results, we present a simpler application of the basic idea behind the proof of Theorem~\ref{THM - 1statement}. We are grateful to Matt Kahle for asking the second author whether
our method could be applied in this way.

We consider the threshold for vanishing of the fundamental group of $X^2(n,p)$.
In \cite{BabsonHoffmanKahle:SimpleConnectivityRandom2Complexes}, Babson, Hoffman and Kahle show a lower bound of order $O(n^{-\epsilon}/n^{1/2})$ for any $\epsilon >0$ as well as an an upper bound of $\sqrt{(log(n)+\varphi(n))/n}$ for any function $\varphi$ with $\varphi(n) \rightarrow \infty$ as $n\rightarrow\infty$;
as mentioned above, the proof of the upper bound is based on studying the random graphs that appear as common 
links of pairs of vertices, which are connected for $p$ in this regime. Using a slightly more involved argument based on 
these random graphs, we show that the logarithmic factor in the upper bound is unnecessary:

\begin{theorem}\label{THM - pi_1}
There is a constant $c>0$ such that the random $2$-complex $X^2(n,p)$ with $p \geq \sqrt{c/n}$ a.a.s.~has a trivial fundamental group: $\pi_1(X^2(n,p))=0$.
\end{theorem}

An upper bound of $\frac{1}{2\sqrt{n}}$ on this threshold for vanishing of the fundamental group has recently been shown by Kor{\'a}ndi, Peled and Sudakov in \cite{KorandiPeledSudakov:2015}.
In \cite{AntoniukLuczakSwiatkowski:2014}, Antoniuk, {\L}uczak and {\'S}wi{\polhk{a}}tkowski use an idea similar to the one presented here to study the triviality of random triangular groups.
}


%

\section{Preliminaries}
\paragraph{Abstract Simplicial Complexes.}
A (finite abstract) \emph{simplicial complex} is a finite set system that is closed under taking subsets, i.e., $F \subset H \in X$ implies $F \in X$.
The sets $F \in X$ are called \emph{faces} of $X$. The \emph{dimension} of a face $F$ is $\dim(F)=|F|-1$. The dimension of $X$ is the maximal dimension of any face. A $k$-dimensional simplicial complex will also be called a \emph{$k$-complex}.

We denote the set of $i$-dimensional faces by $X_i$. The \emph{$i$-skeleton} of $X$ is the simplicial complex $X_{-1} \cup X_0 \cup \ldots \cup X_i$.
A \emph{vertex} of $X$ is a $0$-dimensional face $\{v\}$, the singleton set will be identified with its element $v$. The set of vertices $X_0$, also denoted by $V = V(X)$, is called the \emph{vertex set} of $X$.
Corresponding to the notation $G=(V,E)$ for graphs, we write $2$-complexes as $X=(V,E,T)$ with $E=X_1$ and $T=X_2$.

The random $k$-complex $X^k(n,p)$ has vertex set $V = [n]$, a \emph{complete $(k-1)$-skeleton}, i.e., $X_i = \binom{[n]}{i+1}$ for $i < k$, and every $F \in \binom{[n]}{k+1}$ is added to $X$ independently with probability $p$, which may be constant or, more generally, a function $p(n)$ depending on $n$.
The \emph{complete $k$-complex} $K_t^k$ has vertex set \change{$V = [n]$}{$V = [t]$} and \change{$X_i=\binom{[n]}{i+1}$}{$X_i=\binom{[t]}{i+1}$} for all $i \leq k$.
The \emph{link} $\lk_X(F)$ of a face $F \in X$ in $X$ is the complex \change{$\{H \in X: H \cup F \in X\}$}{$\{H \in X: H \cup F \in X, H\cap F=\emptyset\}$}.

\paragraph{Geometric Simplicial Complexes.}
There is a more geometric way to define simplicial complexes: A \emph{geometric simplex} $\sigma$ is the convex hull of a set of affinely independent points, the \emph{vertices} of $\sigma$, in some Euclidean space $\R^m$. The convex hull of a subset of the vertices of $\sigma$ is a \emph{face} of $\sigma$. A \emph{geometric simplicial complex} is then a finite collection $\Delta$ of geometric simplices in $\R^m$ satisfying two conditions: If $\sigma$ is in $\Delta$ and $\tau$ is a face of $\sigma$, then $\tau$ is also in $\Delta$. Furthermore, the intersection of any two simplices in $\Delta$ is a common face of both, or empty. 

A geometric simplicial complex $\Delta$ defines a topological space, its \emph{polyhedron}, the union of all its simplices: $\|\Delta\| = \bigcup_{\sigma \in \Delta}\sigma \subset \R^m$. It carries the subspace topology inherited from the ambient Euclidean space $\R^m$.
We call $\Delta$ a \emph{triangulation} of $\|\Delta\|$.
Any geometric simplicial complex $\Delta$ gives rise to an abstract complex $X$ in a straight-forward way: A set of vertices forms an (abstract) simplex in $X$ iff it is the vertex set of a geometric simplex in $\Delta$. The geometric complex $\Delta$ is then called a \emph{geometric realization} of $X$, or of any abstract complex isomorphic to it. Here, two simplicial complexes are \emph{isomorphic} if there is a face-preserving bijection between their vertex sets.
 
Any abstract complex $X$ has a geometric realization, e.g., as a subcomplex of a simplex of sufficiently high dimension. We denote by $\|X\|$ the polyhedron of any geometric realization of $X$. This is well-defined because the polyhedra of (geometric realizations of) two isomorphic complexes are homeomorphic (see, e.g., \cite{Matousek:2003}). Also the abstract complex $X$ is called a triangulation of $\|X\|$.
\begin{figure}[htbp]
  \centering\includegraphics[scale=.9]{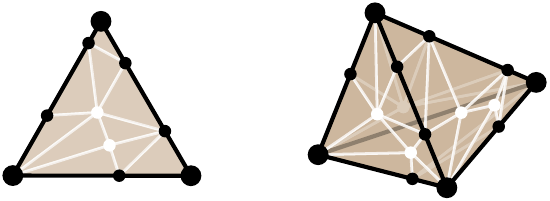}
  \caption{Subdivisions of $K_3^2$ (a \emph{triangulated disk}) and of $K_4^2$. Vertices and edges internal to triangles are drawn in white.\label{fig:subdivision}}
\end{figure}
\paragraph{Subcomplexes and Subdivisions.}
A \emph{subcomplex} of $X$ (or $\Delta$) is a subset $Y \subset X$ ($\Delta$) that is itself a simplicial complex. A \emph{subdivision} of a geometric simplicial complex $\Delta$ is a complex $\Delta'$ with $\|\Delta\|=\|\Delta'\|$ such that every simplex of $\Delta'$ is contained in some simplex of $\Delta$. For an abstract complex $X$, a complex $X'$ is a subdivision of $X$ if there exist geometric realizations $\Delta$ and $\Delta'$ of $X$ and $X'$ such that $\Delta'$ is a subdivision of $\Delta$.
A subdivision of an abstract $2$-complex $X$ can be seen as a $2$-complex $X'$ that is obtained by replacing the edges of $X$ with internally-disjoint paths and the triangles of $X$ with internally-disjoint triangulated disks such that for every triangle the subdivision of the triangle agrees with the subdivisions of its edges. See Figure~\ref{fig:subdivision} for an illustration.

We say that a complex $X$ has a fixed complex $K$ as a \emph{topological minor} if $X$ contains a copy of some subdivision of $K$.

%
\section{On the Threshold for Vanishing Fundamental Group}\label{Sec - pi_1}

\add{\protect Before we come to the proof of Theorem~\ref{THM - 1statement}, the upper bound for the threshold for topological minor containment, we present a simpler application of its basic idea and prove Theorem~\ref{THM - pi_1}.

The fundamental group $\pi_1(Y,y_0)$ of a topological space $Y$ at a point $y_0\in Y$ is the group of homotopy classes of loops, i.e., continuous maps $f:I\rightarrow Y$ with the same starting and ending point $f(0) = f(1) = y_0$. The fundamental group of a path-connected space is independent, up to isomorphism, of the choice of basepoint $y_0$ and is then denoted by $\pi_1(Y)$. (See, e.g., \cite{Hatcher:2002} for more details.)

For a simplicial complex $X$ the fundamental group $\pi_1(X):=\pi_1(\|X\|)$ depends only on the $2$-skeleton of $X$, and it suffices to consider loops on the $1$-skeleton (``edge paths''). (See, e.g., \cite{Spanier:1981}.)

To prove Theorem~\ref{THM - pi_1}, we will show that in $X^2(n,p)$ with $p\geq\sqrt{c/n}$ (for $c>0$ sufficiently large) a.a.s.~every cycle consisting of three edges can be filled with a subdivision of a disk. This implies that every $3$-cycle (and, because the $1$-skeleton is the complete graph, also every other cycle) is homotopically trivial. As the property of having a trivial fundamental group is monotone (preserved under adding simplices), it is enough to consider the case $p=\sqrt{c/n}$.

Consider three vertices $a,b,c \in [n]$. To fill the cycle spanned by $a$,$b$ and $c$, we find a fourth vertex $x$ and fill the three cycles $\{a,b,x\}$,$\{a,c,x\}$ and $\{b,c,x\}$.

\begin{figure}[htbp]
  \centering\includegraphics[scale=.8]{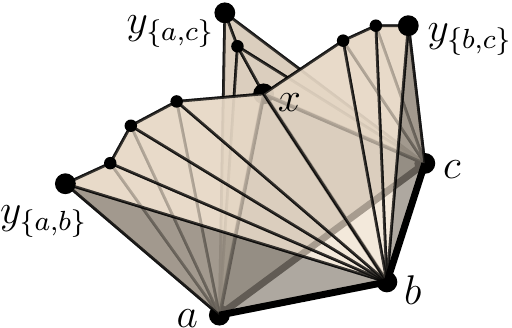}
  \caption{A filling of $\{a,b,c\}$ given by a vertex $x$ good for $\{a,b\}$,$\{a,c\}$ and $\{b,c\}$ \label{fig:triangle}}
\end{figure}
For a pair of vertices $a,b \in [n]$, we consider the graph $G_{a,b} := \lk_X(a) \cap \lk_X(b)$. We call a vertex $x$ \emph{``good''} for the pair $\{a,b\}$ if there is a vertex $y$ with $\{a,b,y\} \in X$ such that $x$ and $y$ are contained in the same connected component of $G_{a,b}$. Note that for a good vertex $x$, any path connecting $x$ and $y$ in $G_{a,b}$ together with the triangle $\{a,b,y\}$ forms a subdivision of a disk filling the cycle $\{a,b,x\}$.
Hence, finding a vertex $x$ that is good for $\{a,b\}$,$\{a,c\}$ and $\{b,c\}$ gives a filling of the cycle $\{a,b,c\}$. See Figure~\ref{fig:triangle} for an illustration. (Note that we do not have to (and possibly will not) find disjoint fillings of the three cycles $\{a,b,x\}$, $\{a,c,x\}$ and $\{b,c,x\}$ as depicted in the figure.)
}

\add{\protect
We denote the set of vertices that are good for $\{a,b\}$ by $S_{a,b}$. In $X^2(n,p)$ we can ensure that for any pair of vertices this set of good vertices is large with high probability:
\begin{lemma}\label{Lemma:SizeGoodSets}
There exist constants $c,\beta>0$ such that in $X^2(n,p)$ with $p=\sqrt{c/n}$
\[
\Pr\big[|S_{a,b}| >\tfrac{2}{3}n\big] \geq 1 - e^{-\beta n^{1/3}}
\]
 for any $a,b \in [n]$.
\end{lemma}

Our previous observations will allow us to reduce the problem to studying the graphs $G_{a,b}$, which are random graphs of type $G(n-2,p^2)$. We will employ the following classical result by Erd\H{o}s and R\'{e}nyi \cite{ErdosRenyi:1960} on the phase transition of the random graph $G(n,p)$. For a simple proof see \cite{KrivelevichSudakov:2012}, in particular Section 3.1 describing how to achieve the bound on the probability.

\begin{theorem}[\cite{ErdosRenyi:1960,KrivelevichSudakov:2012}]\label{THM:PhaseTransition}
For every  $\gamma >0$ there is $\kappa = \kappa(\gamma) >0$ such that with probability at least $1-e^{-\kappa n^{1/3}}$ the random graph $G(n,\frac{1 + \gamma}{n})$ has a connected component of size at least $\frac{\gamma}{2} n$.
\end{theorem}

We will furthermore use the following classical result on the concentration of binomially distributed random variables:

\begin{theorem}[{Chernoff's inequality, see, e.g., \cite[Theorem~1]{Janson:2002}}]\label{Chernoff}
Let $X$ be a binomially distributed random variable with parameters $n$ and $p$. Then for any $t \geq 0$:
$$ \Pr[X \geq \E[X] + t] \leq e^{-\frac{t^2}{2(np +t/3)}} \quad \text{ and }\quad
 \Pr[X \leq \E[X] - t] \leq e^{-\frac{t^2}{2np}}.
$$
\end{theorem}

\begin{proof}[Proof of Lemma~\ref{Lemma:SizeGoodSets}]
Fix $a,b \in [n]$. Denote by $C_{a,b}$ the largest connected component of the graph $G_{a,b} = \lk_X(a) \cap \lk_X(b)$. If there are several components of maximum size, let $C_{a,b}$ be the one containing the smallest vertex.
Note that if there is a vertex $y \in C_{a,b}$ such that $\{a,b,y\} \in X$, then every vertex in $C_{a,b}$ is good for $\{a,b\}$, i.e., $C_{a,b}\subseteq S_{a,b}$.
Hence, it suffices to consider the following event:
\[
\mathcal{A}_{a,b}=\big\{X\subseteq K_n^k:K_n^{k-1} \subseteq X, \exists y \in C_{a,b}: \{a,b,y\} \in X \text{ and } |C_{a,b}| >\tfrac{2}{3} n\big\}.
\]

For a fixed set $S\subset V\setminus\{a,b\}$ with $|S|> \tfrac{2}{3}n$ the number of vertices $y \in S$ with $\{a,b,y\} \in X$ is binomially distributed with parameters $|S|$ and $p$. Thus, its expectation is $|S|p$ and by Chernoff's inequality, Theorem~\ref{Chernoff}:
\[
\Pr\big[|\{ y \in S: \{a,b,y\} \in X\}|=0\big] \leq e^{-\frac{|S|p}{2}}\leq e^{-\sqrt\frac{cn}{9}}.
\]
We now condition on $C_{a,b}$ being $S$:
\[
\begin{split}
\Pr\big[\mathcal{A}_{a,b}\big] &= \sum_{S \subset V\setminus\{a,b\}} \Pr\big[\mathcal{A}_{a,b}\,\big|\,C_{a,b}=S\big]\cdot \Pr[C_{a,b} = S]\\
&=\sum_{|S|> \tfrac{2}{3}n} \Pr\big[ \exists y \in C_{a,b}: \{a,b,y\} \in X\,\big|\,C_{a,b}=S\big]\cdot \Pr[C_{a,b} = S]\\
&\geq \Big(1-e^{-\sqrt\frac{cn}{9}}\Big)\cdot\sum_{|S|> \tfrac{2}{3}n} \Pr[C_{a,b} = S]= \Big(1-e^{-\sqrt\frac{cn}{9}}\Big)\cdot \Pr[ |C_{a,b}| > \tfrac{2}{3} n].
\end{split}
\]
Recall that $C_{a,b}$ is the largest component of the graph $G_{a,b}$, which is a random graph of type $G(n-2,p^2)$. As $p^2=c/n$, we can now employ Theorem~\ref{THM:PhaseTransition}: If we choose $c$ large enough, there is $\kappa>0$ such that $C_{a,b}$ has size greater than $\tfrac{2}{3}n$ with probability $1-e^{-\kappa n^{1/3}}$. 
Then we have:
\[
\begin{split}
\Pr\big[\mathcal{A}_{a,b}\big] &\geq \Big(1-e^{-\sqrt\frac{cn}{9}}\Big)(1-e^{-\kappa n^{1/3}})= 1-e^{-\sqrt\frac{cn}{9}}-\Big(1-e^{-\sqrt\frac{cn}{9}}\Big)e^{-\kappa n^{1/3}}\\
 & \geq 1-e^{-\sqrt\frac{cn}{9}}-e^{-\kappa n^{1/3}} \geq 1-e^{-\log(2)\kappa n^{1/3}}.
\end{split}
\]
\end{proof}

Now, we know that with high probability the set of good vertices is large for any pair $a,b$ and use this to show that we can fill every $3$-cycle.
\begin{proposition}
There is $c>0$ such that in $X^2(n,p)$ with $p=\sqrt{c/n}$ a.a.s.~every cycle consisting of three edges can be filled with a subdivision of a disk.
\end{proposition}

\begin{proof}
We observed above that finding a vertex $x$ that is good for $\{a,b\}$,$\{a,c\}$ and $\{b,c\}$ gives a filling of the cycle $\{a,b,c\}$. It hence suffices to show that the probability
\[
\Pr\big[\forall a,b,c\in [n]: S_{a,b}\cap S_{a,c} \cap S_{b,c} \neq\emptyset\big]
\]
converges to $1$.
For fixed $a,b,c \in [n]$ we can employ Lemma~\ref{Lemma:SizeGoodSets}:
\[\begin{split}
\Pr\big[S_{a,b}\cap S_{a,c} \cap S_{b,c} =\emptyset\big] &\leq \Pr\big[|S_{a,b}| \leq \tfrac{2}{3}n, |S_{a,c}| \leq \tfrac{2}{3}n \text{ or } |S_{b,c}| \leq \tfrac{2}{3}n\big]\\
& \leq \Pr\big[|S_{a,b}| \leq \tfrac{2}{3}n\big]+\Pr\big[|S_{a,c}| \leq \tfrac{2}{3}n\big]+\Pr\big[|S_{b,c}| \leq \tfrac{2}{3}n\big]\\
& \leq 3e^{-\beta n^{1/3}}.
\end{split}
\]

Using a union bound, we see that
\[
\begin{split}
\Pr\big[\forall a,b,c\in [n]: S_{a,b}\cap S_{a,c} \cap S_{b,c} \neq\emptyset\big] &\geq 1-\tbinom{n}{3}\cdot \Pr\big[S_{a,b}\cap S_{a,c} \cap S_{b,c} =\emptyset\big]\\
&\geq 1-3e^{-\beta n^{1/3}}.
\end{split}
\]

\end{proof}
}
\section{The Upper Bound for Topological Minor Containment}\label{Sec - 1statement}
We now address Theorem~\ref{THM - 1statement}. For fixed $t\geq k+1$, we aim to find a copy of a subdivision of $K_t^k$ in $X^k(n,p)$. We would be allowed to subdivide faces of $K_t^k$ of any dimension, but there will be no need for this:  we find $t$ vertices and take all faces of dimension at most $k-1$ spanned by these vertices to form the $(k-1)$-skeleton of our subdivision of $K_t^k$. We then show that the $k$-spheres (boundaries of $k$-simplices) spanned by the $(k-1)$-faces between any $k$ of them can be filled with disjoint triangulated $k$-simplices.

\paragraph{Basic Set-Up.}
We will only consider $k$-complexes with vertex set $[n]$ and complete $(k-1)$-skeleton. For notational convenience, we assume without loss of generality that $n$ is divisible by $2\binom{t}{k+1}$.
Fix a partition of the vertex set $V=[n]$ into two sets $U$ and $W$, each of size $\frac{n}{2}$. We will choose the $t$ vertices of $K_t^k$ from $U$, whereas the internal vertices for fillings will come from $W$. To ensure disjointness of the fillings of different $k$-spheres, we partition $W$ into $\binom{t}{k+1}$ sets $W_\sigma$, $\sigma \in \binom{[t]}{k+1}$, each of size $n/(2\binom{t}{k+1})$, and choose the internal vertices of the filling for each $\sigma\in \binom{[t]}{k+1}$ from $W_\sigma$.

For a $(k-1)$-face $F \in \binom{[n]}{k}$, denote by $G_F$ the graph $\bigcap_{H \in \binom{F}{k-1}} \lk_X(H)$ which has vertex set $[n]\!\setminus\!F$ and edge set $\{e \subset [n]\!\setminus\!F: e\cup H \in X \text{ for all } H \in \binom{F}{k-1}\}$. Denote by $C_F^\sigma$ the largest connected component of $G_F[W_\sigma]$. If there are several components of maximum size, let $C_F^\sigma$ be the one containing the smallest vertex.

\paragraph{The Main Idea.}
The basic idea of the proof is the following lemma, based on an idea going back at least to Brown, Erd\H{o}s and S\'os \cite{Brown:1973ve} that is also used in \cite{BabsonHoffmanKahle:SimpleConnectivityRandom2Complexes}.
\begin{figure}[htbp]
  \centering\includegraphics[scale=.8]{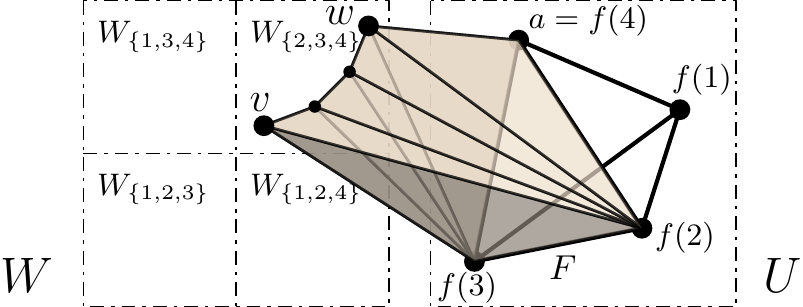}
  \caption{Part of a subdivision of a $K_4^2$: Filling of $f(\sigma)$ for $\sigma = \{2,3,4\}$ \label{fig:containment}}
\end{figure}
\begin{lemma}\label{Lemma:BasicIdea}
Let $X$ be a $k$-complex with vertex set $[n]$ and complete $(k-1)$-skeleton.
Suppose there is a set $A \subset U$, $|A|=t$ with a bijection $f\!\!:\!\![t]\rightarrow A$ satisfying the following property: For every $k$-face $\sigma \in \binom{[t]}{k+1}$ of $K_t^k$ there is a vertex $a \in f(\sigma)$ such that for $F=f(\sigma)\!\setminus\!\{a\}$ there are vertices $v,w \in C_F^\sigma$ with $F \cup \{v\} \in X$ and $\{a,w\} \in G_F$.
Then $X$ contains a subdivision of $K_t^k$.
\end{lemma}

\begin{proof}
For $\sigma \in \binom{[t]}{k+1}$ and $a$, $F$, $v$, $w$ as above there exists a path in $G_F[W_\sigma]$ between $v$ and $w$ which together with the $k$-faces $\{a,w\} \cup H$ for $H \in \binom{F}{k-1}$  and the $k$-face  $F \cup \{v\}$, creates a $k$-dimensional disk filling the $k$-sphere (boundary of a $k$-simplex) created by the $(k-1)$-faces $F' \subset f(\sigma)$, $|F'|=k$. By choosing a distinct $W_\sigma$ for each $\sigma \in \binom{[t]}{k+1}$ we ensure disjoint fillings.
 See Figure~\ref{fig:containment} for an illustration for the case $k=2$.
\end{proof}

\paragraph{Random Complexes.}
We now proceed to show that for a suitable constant $c=c(t,k)$ and $p\geq \sqrt[k]{c/n}$ the random complex $X^k(n,p)$ a.a.s.~satisfies the conditions of Lemma~\ref{Lemma:BasicIdea}.
 We first give a criterion for complexes satisfying these conditions  and then show that this criterion is satisfied a.a.s.~by a random complex.

For fixed $F \subset U$, $|F|=k$ and $\sigma \in \binom{[t]}{k+1}$, call $u \in U \setminus F$ \emph{connected to $C_F^\sigma$} if $\{u,w\} \in G_F$ for some $w \in C_F^\sigma$ and let $N_F^\sigma = \{u \in U\!\setminus\!F\!: u \text{ connected to } C_F^\sigma\}$.
Consider two families of $k$-complexes $X$ with vertex set $[n]$ and complete $(k-1)$-skeleton:
\begin{itemize}
\item \change{$\mathcal{A}_{F,\sigma} = \{X\subseteq \binom{[n]}{k+1}: \binom{[n]}{k} \subseteq X, \exists v \in C_F^\sigma \text{ with }F \cup \{v\} \in X\}$}{\protect$\mathcal{A}_{F,\sigma} = \{X\subseteq K_n^k: K_n^{k-1} \subseteq X, \exists v \in C_F^\sigma \text{ with }F \cup \{v\} \in X\}$}.
\item For $\delta >0$: \change{$\mathcal{B}_{F,\sigma,\delta} = \{X\subseteq \binom{[n]}{k+1}: \binom{[n]}{k} \subseteq X, |N_F^\sigma| \geq (1-\delta)(|U|-k)\}$}{\protect$\mathcal{B}_{F,\sigma,\delta} = \{X\subseteq K_n^k: K_n^{k-1} \subseteq X, |N_F^\sigma| \geq (1-\delta)(|U|-k)\}$}.
\end{itemize}

\begin{lemma}\label{Lemma:ExistenceGoodSet}
Let $X$ be a $k$-complex with vertex set $[n]$ and complete $(k-1)$-skeleton.
If there is a $\delta<1\bigl/\bigl(\binom{t}{k+1}(k+1)\bigr)$ such that $X \in \mathcal{A}_{F,\sigma} \cap \mathcal{B}_{F,\sigma,\delta}$ for all $F \subset U$, $|F|=k$ and  $\sigma \in \binom{[t]}{k+1}$,
then there exists a set $A\!\in\!\binom{U}{t}$ satisfying the conditions of Lemma~\ref{Lemma:BasicIdea}.
\end{lemma}

\begin{proof}
 We need to show the existence of a set $A \subset U$, $|A|=t$ with a bijection $f\!\!:\!\![t]\rightarrow A$ such that for every $\sigma \in \binom{[t]}{k+1}$ there is a vertex $a \in f(\sigma)$ such that for $F=f(\sigma)\!\setminus\!\{a\}$:
\begin{enumerate}
 \item There is $v \in C_F^\sigma$ with $F \cup \{v\} \in X$.
 \item The vertex $a$ is connected to $C_F^\sigma$. 
\end{enumerate}
As $X \in \mathcal{A}_{F,\sigma}$ for all $F$ and $\sigma$, the first condition holds for any choice of $A$,$f$,$\sigma$ and $a$. So we only need to deal with the second condition.
We consider tupels $(a_1,a_2,\ldots,a_t)$ with $a_i \in U$ and all $a_i$ pairwise distinct and let $A = \{a_1,a_2,\ldots,a_t\}$. The function $f$ is then determined by $f(i) = a_i$. We show that for a tupel chosen uniformly at random we have
$$\Pr\left[\exists \sigma\in \tbinom{[t]}{k+1}, a \in f(\sigma): a \text{ not connected to } C_{f(\sigma)\!\setminus\!\{a\}}^\sigma\right] < 1.$$
Thus, there is a tuple that also satisfies the second condition.
For fixed $\sigma$ and $j\in\sigma$:
\begin{equation*}
\begin{split}
 \Pr&\left[f(j) \text{ not connected to } C_{f(\sigma\!\setminus\!\{j\})}^\sigma\right]\\
&= \sum\nolimits_{F \in \binom{U}{k}} \Pr\left[f(j) \notin N_F^\sigma\,|\,f(\sigma\!\setminus\!\{j\})=F\right]\cdot\Pr\left[f(\sigma\!\setminus\!\{j\})=F\right]\\
&= \tbinom{|U|}{k} \cdot \delta \cdot \tfrac{1}{\binom{|U|}{k}} = \delta
\end{split}
\end{equation*}
By a union bound, we hence have 
$$\Pr\left[\exists \sigma, a \in f(\sigma): a \text{ not connected to } C_{f(\sigma)\!\setminus\!\{a\}}^\sigma\right] \leq \tbinom{t}{k+1}(k+1)\delta < 1.$$

\end{proof}
Now, we finally turn to random complexes $X^k(n,p)$.
As the property of containing \change{a fixed subcomplex}{\protect a given topological minor} is monotone (preserved under adding simplices), it is enough to consider the case $p=\sqrt[k]{c/n}$.
\begin{lemma}
For every $k\geq 2$ and $t \geq k+1$ there is a constant $c=c(t,k)>0$ such that for $p= \sqrt[k]{\frac{c}{n}}$ the random complex $X^k(n,p)$ asymptotically almost surely satisfies the conditions of Lemma~\ref{Lemma:ExistenceGoodSet}.
\end{lemma}
\begin{proof}
Let $k\geq 2$ and $t \geq k+1$. Let $T = \binom{t}{k+1}$.  
We show that there is $c>2T$ and $\delta<1/(T(k+1))$ such that $X^k(n,p)$ for $p=\sqrt[k]{c/n}$ a.a.s.~satisfies the conditions of Lemma~\ref{Lemma:ExistenceGoodSet}, i.e.,
$$X^k(n,p)\in \bigcap_{F,\sigma} \mathcal{A}_{F,\sigma} \cap \mathcal{B}_{F,\sigma,\delta}.$$

Fix $F \subset U$, $|F|=k$ and  $\sigma \in \binom{[t]}{k+1}$.
The probability of the events  $\mathcal{A}=\mathcal{A}_{F,\sigma}$ and $\mathcal{B}=\mathcal{B}_{F,\sigma,\delta}$ depends on the size of $C_F^\sigma$, the largest connected component of the graph $G_F[W_\sigma]$, which is a random graph of type $G(|W_\sigma|,p^k)$.

As $p^k=c/n = \frac{c/\left(2T\right)}{|W_\sigma|}$, for $c$ large enough the graph $G_F[W_\sigma]$ \emph{fails} to have a giant component of size linear in $|W_\sigma|$ with exponentially small probability:
\remove{For every  $\gamma >0$ a random graph $G(n,\frac{1 + \gamma}{n})$ has a connected component of size at least $\frac{\gamma^2n}{5}$ with probability $1-e^{-\kappa n}$ for some $\kappa = \kappa(\gamma) >0$ (see e.g,~\protect \cite{KrivelevichSudakov:2012}).}
\change{So}{\protect By Theorem~\ref{THM:PhaseTransition}} for any $c > 2T$ there are $\epsilon>0$ and $\kappa >0$ such that
\change{\protect
 $$
\Pr[|C_F^\sigma| \geq (1-\epsilon)|W_\sigma|] \geq 1- e^{-\kappa n}.
$$}
{\protect
$$
\Pr[|C_F^\sigma| \geq (1-\epsilon)|W_\sigma|] \geq 1- e^{-\kappa n^{1/3}}.
$$}
As we will later need that $\delta > e^{-\frac{c(1-\epsilon)}{2T}}$, we choose $c > \frac{\ln(T) + \ln(k+1)}{1-\epsilon} \cdot 2T> 2T$. Then $e^{-\frac{c(1-\epsilon)}{2T}} < \frac{1}{T(k+1)}$ and we choose $\delta$ with $e^{-\frac{c(1-\epsilon)}{2T}} < \delta < \frac{1}{T(k+1)}$.

For  $S \subset W_\sigma$ denote by $\Pr_S$ the conditional probability when conditioning on $C_F^\sigma = S$.
Then $\Pr[\mathcal{A}_{F,\sigma} \cap \mathcal{B}_{F,\sigma,\delta}]$ is at least
$$
\sum_{S} \Pr\nolimits_S\left[\mathcal{A}_{F,\sigma} \cap \mathcal{B}_{F,\sigma,\delta}\right] \cdot \Pr\left[C_F^\sigma = S\right],
$$
where the sum runs over all $S \subset W_\sigma$ with $|S| \geq (1-\epsilon)|W_\sigma|$.

As $\mathcal{A}_{F,\sigma}$ and $\mathcal{B}_{F,\sigma,\delta}$ depend on different kinds of $k$-faces and the presences of $k$-faces are 
decided independently, we have 
$$\Pr\nolimits_S\left[\mathcal{A}_{F,\sigma} \cap \mathcal{B}_{F,\sigma,\delta}\right] = \Pr\nolimits_S\left[\mathcal{A}_{F,\sigma}\right] \cdot \Pr\nolimits_S\left[\mathcal{B}_{F,\sigma,\delta}\right].$$
We consider the two terms seperately:

\leftskip=.4cm
\paragraph{$\Pr\nolimits_S\left[\mathcal{A}_{F,\sigma}\right]$:} Here we consider $\Pr\left[\exists v \in S \text{ with } F\cup\{v\} \in X\right]$. The number $f(X)$ of vertices $v \in S$ with $F\cup\{v\} \in X$ is a binomially distributed variable with parameters $|S|$ and $p$. Hence, its expectation is $|S|p$ and by Chernoff's inequality, Theorem~\ref{Chernoff}: $$\Pr\left[f(X) = 0\right] \leq e^-{\frac{|S|p}{2}} \leq e^{-(1-\epsilon)\frac{\sqrt[k]{c}}{4T}n^{1-1/k}}.$$

\paragraph{$\Pr\nolimits_S\left[\mathcal{B}_{F,\sigma,\delta}\right]$:} Call $u \in U\!\setminus\!F$ \emph{connected to $S$} if $\{u,w\} \in G_F$ for some $w \in S$. Then we need to consider 
$$\Pr\left[|\{u \in U\!\setminus\!F : u \text{ connected to } S\}| \geq (1-\delta)(|U|-k)\right].$$
For fixed $u \in U\!\setminus\!F$ the probability \emph{not} to be connected to $S$ is $\lambda = (1-p^k)^{|S|} \leq e^{-p^k|S|} \leq e^{-\frac{c(1-\epsilon)}{2T}}$.
For each $u$ the decisions over the $k$-faces deciding whether $u$ is connected to $S$ are taken independently. Hence, also the number $g(X)$ of vertices $u \in U\!\setminus\!F$ that are connected to $S$ is a binomially distributed variable with parameters $(|U|-k)$ and $(1-\lambda)$.
As we chose $\delta > e^{-\frac{c(1-\epsilon)}{2T}} \geq \lambda$, we get by Chernoff's inequality, Theorem~\ref{Chernoff}, for large enough $n$:
\begin{equation*}
\begin{split}
 \Pr\left[g(X) < (1-\delta)(|U|-k)\right] &= \Pr\left[g(X) < \E[g(X)] - (\delta - \lambda)(|U|-k)\right]\\
 &\leq e^{-\frac{(\delta - \lambda)^2(|U|-k)}{2(1-\lambda)}} \leq e^{-\frac{(\delta - \lambda)^2}{5}n}.
 \end{split}
\end{equation*}

\leftskip=0cm
Notice that the probabilities $\Pr\nolimits_S\left[\mathcal{A}_{F,\sigma}\right]$ and $\Pr\nolimits_S\left[\mathcal{B}_{F,\sigma,\delta}\right]$ do not depend on $S$.
Hence we can use $\sum_S \Pr\left[C_F^\sigma = S\right]=\Pr\left[|C_F^\sigma| \geq (1-\epsilon)|W_\sigma|\right]$ and get by the choice of $c$ and $\epsilon$:
\change{\protect
$$  \Pr[\mathcal{A}_{F,\sigma} \cap \mathcal{B}_{F,\sigma,\delta}] \geq \left(1-e^{-(1-\epsilon)\frac{\sqrt[k]{c}}{4T}n^{1-1/k}}\right)\left(1-e^{-\frac{(\delta - \lambda)^2}{5}n}\right)\left(1-e^{-\kappa n}\right) \geq 1 - e^{-\beta n^{1-1/k}}$$}
{\protect
\[\begin{split}
 \Pr[\mathcal{A}_{F,\sigma} \cap \mathcal{B}_{F,\sigma,\delta}] &\geq \left(1-e^{-(1-\epsilon)\frac{\sqrt[k]{c}}{4T}n^{1-1/k}}\right)\left(1-e^{-\frac{(\delta - \lambda)^2}{5}n}\right)\left(1-e^{-\kappa n^{1/3}}\right)\\
&\geq 1-e^{-(1-\epsilon)\frac{\sqrt[k]{c}}{4T}n^{1-1/k}}-e^{-\frac{(\delta - \lambda)^2}{5}n}-e^{-\kappa n^{1/3}}\geq 1 - 3e^{-\kappa n^{1/3}}.
 \end{split}
 \]}
\remove{for some $\beta>0$.} Applying a union bound, we get for some $\alpha > 0$:
\change{\protect
$$
\Pr\left[\exists F \subset U,|F|=k, \sigma \in \tbinom{[t]}{k+1}: \neg \mathcal{A}_{F,\sigma} \cup \neg \mathcal{B}_{F,\sigma,\delta}\right] \leq \tbinom{n/2}{k}\cdot T \cdot e^{-\beta n^{1-1/k}} \leq e^{-\alpha n^{1-1/k}}.
$$}
{\protect
$$
\Pr\left[\exists F \subset U,|F|=k, \sigma \in \tbinom{[t]}{k+1}: \neg \mathcal{A}_{F,\sigma} \cup \neg \mathcal{B}_{F,\sigma,\delta}\right] \leq \tbinom{n/2}{k}\cdot T \cdot 3e^{-\kappa n^{1/3}} \leq e^{-\alpha n^{1/3}}.
$$}
\end{proof}
\section{The Lower Bound for Topological Minor Containment}\label{Sec - 0statement}

We now turn to  the proof of Theorem~\ref{THM - 0statement} on random $2$-complexes $X^2(n,p)$. Our goal is to show the existence of a constant $c\leq 1$  such that for $p=\sqrt{c/n}$ the probability to find a subdivision of $K^2_t$ converges to zero.
The proof bases on the following simple observation: If a complex contains a subdivision of $K_t^2$ with $t\geq 10$, it also contains a subdivision of a triangulation of $\Sigma_2$, the orientable surface of genus $2$\footnote{The smallest possible triangulation of $\Sigma_2$ has $10$ vertices, see \cite{Huneke:1978}.}. We then use that the number of triangulations of any surface with a fixed number $l$ of vertices is known to be at most simply exponential in $l$.

Bounds on the number of triangulations of a fixed closed surface can be drawn from the theory of enumeration of maps on surfaces which has its beginning in Tutte's famous results on the number of rooted maps on the sphere \cite{Tutte:1963, Tutte:1968, Tutte:1973}.
As the terminology in these references differs a lot from ours and as furthermore the classes of objects that are counted are not exactly the same, we first explain in detail the enumeration result we will use.
We rely on \cite{BenderCanfield:1986,BenderCanfieldRobinson:1988,BenderRichmond:1986,Gao:1991}.

\paragraph{Maps on Surfaces.}
Let $S$ be a connected compact $2$-manifold without boundary. A \emph{map} $M=(S,G,\Phi)$ on $S$ is a graph $G$ together with an embedding $\Phi$ of $G$ into $S$ such that each connected component of $S \setminus \Phi(G)$ is simply connected, i.e., each face is a disk.
Graphs are unlabeled, finite and connected, loops and multiple edges are allowed.

A map is \emph{rooted} if an edge, a direction along the edge and a side of the edge are distinguished.
An edge is called \emph{double} if its image belongs to the boundary of only one face. Any other, \emph{single}, edge belongs to two faces.
The \emph{valency} of a face is the number of single edges in its boundary plus twice the number of double edges.
A \emph{triangular map} is a map such that each face has valency three.

Two maps $(S,G,\Phi)$ and $(S',G',\Phi')$ are considered \emph{equivalent} if there is a homeomorphism $h\!:\!S \rightarrow S'$ and a graph isomorphism $g\!:\!G \rightarrow G'$ such that $h\Phi=\Phi'g$.

\paragraph{Triangular Maps vs. Triangulations.}
Let $M=(S,G,\Phi)$ be a triangular map such that the graph $G=(V,E)$ is simple, i.e., does not have loops or multiple edges. Then every face of $M$ has a boundary consisting of exactly three edges. Define a $2$-complex $X(M) =(V,E,T(M))$ by letting 
$$T(M) := \{\{u,v,w\}: u,v,w \in V \text{ are the vertices of a face of }M\}.$$
Equivalent maps yield isomorphic complexes: 

\begin{lemma}\label{EquivalentMaps}
 Let $M=(S,G,\Phi)$ be a triangular map such that the graph $G=(V,E)$ is simple and let $M'=(S',G',\Phi')$ be equivalent to $M$. Then $X(M)$ and $X(M')$ are isomorphic.
\end{lemma}

\begin{proof}
 Since $M$ and $M'$ are equivalent, there is a homeomorphism $h\!:\!S \rightarrow S'$ and a graph isomorphism $g\!:\!G \rightarrow G'$ such that $h\Phi=\Phi'g$.
We show that $g$ is also an isomorphism between $X(M)$ and $X(M')$. As $g$ is a graph isomorphism, all we need to show is that $g$ preserves $2$-faces. Let $\{u,v,w\} \in T(M)$, so $u,v,w$ are the vertices of a face of $M$. This face is mapped to some disk in $S'$ by $h$. As $h\Phi = \Phi'g$, this disk is a face of $M'$ with $g(u),g(v)$ and $g(w)$ as boundary vertices. Hence, $\{g(u),g(v),g(w)\} \in T(M')$. The same argument shows that any $2$-face of $X(M')$ is mapped to a $2$-face of $X(M)$.
\end{proof}

For a $2$-complex $X=(V,E,T)$ such that $\|X\|$ is homeomorphic to a surface $S$, define a triangular map $M(X) = (S, (V,E), \Phi)$, where $\Phi$ is the restriction of a homeomorphism $\|X\| \rightarrow S$ to the $1$-skeleton of $X$. The following lemma shows that $M(X)$ is well-defined and that isomorphic complexes give rise to equivalent maps:

\begin{lemma}\label{IsomorphicComplexes}
 Let $X=(V,E,T)$ be a $2$-complex such that $\|X\|$ is homeomorphic to a surface $S$ and let $X'=(V',E',T')$ be isomorphic to $X$.
 Let furthermore $\varphi\!:\!\|X\|\rightarrow S$ and $\varphi'\!:\!\|X'\|\rightarrow S$ be homeomorphisms and define $\Phi$ and $\Phi'$ to be the restrictions of $\varphi$ and $\varphi'$ to the $1$-skeleta of $X$ and $X'$, respectively. Then $(S, (V,E), \Phi)$ and $(S, (V',E'), \Phi')$ are equivalent.
\end{lemma}

\begin{proof}
 Let $f\!:\!V(X) \rightarrow V(X')$ be an isomorphism between $X$ and $X'$. Then the affine extension $\|f\|\!:\!\|X\|\rightarrow\|X'\|$ is a homeomorphism (see, e.g., \cite[Proposition~1.5.4]{Matousek:2003}). So $\|X'\|$ is also homeomorphic to $S$. Choosing $g=f$ and $h=\varphi'\|f\|\varphi^{-1}$, we get $h\Phi=\Phi'g$.
\end{proof}

Lemmas~\ref{EquivalentMaps} and \ref{IsomorphicComplexes} show that there is a bijection between equivalence classes of triangular maps with simple underlying graph on a surface $S$ and isomorphism classes of $2$-complexes with polyhedron homeomorphic to $S$.

\paragraph{The Number of Triangulations.}
In \cite{Gao:1991} Gao gives an asymptotic enumeration result for rooted triangular maps on any closed surface.

\begin{theorem}[\add{\protect Gao, \cite{Gao:1991}}]\label{GaosResult}
Let $T_g(l)$ denote the number of $l$-vertex rooted triangular maps on the orientable surface of genus $g$.\footnote{Gao also considers non-orientable surfaces, in which we are not interested here.}
There is a constant $t_g$, independent of $l$, such that for $l \rightarrow \infty$,
\[T_g(l) \sim t_g l^{5(g-1)/2} (12 \sqrt{3})^l.\]
\end{theorem}


We are interested in the number $\tau_g(l)$ of $l$-vertex triangulations of the orientable surface $\Sigma_g$ of genus $g$, i.e., the number of $2$-complexes $X=(V,E,T)$ such that $|V|=l$ and $\|X\|$ is homeomorphic to $\Sigma_g$. By the considerations above this is the number of triangular maps on $\Sigma_2$ with a simple underlying graph. As Gao's result also allows loops and multiple edges and makes a distinction between equivalent maps that are rooted in a different way, we get $\tau_g(l) \leq T_g(l)$ and hence:

\begin{corollary}\label{NumberTriangulations}
 Let $\tau_g(l)$ be the number of triangulations of $\Sigma_g$, the orientable surface of genus $g$, with $l$ vertices. There is a constant $K_g>0$, independent of $l$, such that $\tau_g(l) \leq K_g^l$.
\end{corollary}

\paragraph{Proof of the Lower Bound on the Threshold.}   Now that we have established (Corollary~\ref{NumberTriangulations}) that the number of triangulations of any fixed surface with a fixed number $l$ of vertices is at most simply exponential in $l$, we can turn to the proof of Theorem~\ref{THM - 0statement}.

\begin{proof}[Proof of Theorem~\ref{THM - 0statement}]
Fix $t \in \N$, $t \geq 10$ and let $T_0$ be a triangulation of $\Sigma_2$, the orientable surface of genus $2$,  with $10$ vertices.
As $T_0$ is a subcomplex of $K^2_t$, we have:
\[
\Pr\big[X^2(n,p) \text{ contains a subdiv.~of } K^2_t\big]
\leq \Pr\big[X^2(n,p) \text{ contains a subdiv.~of } T_0\big].
\]
We show that for sufficiently small $p$ the latter probability tends to $0$.  Ignoring that we only consider subdivisions of $T_0$, we get:
\[
\Pr\big[X^2(n,p) \text{ contains a subdivision of } T_0\big]
\!\leq\!\sum_{l=1}^{n} \sum_{T\in\mathcal{T}_l}\!\Pr\big[T\subseteq X^2(n,p)\big],
\]
where the second sum is over the set $\mathcal{T}_l$ of all triangulations of $\Sigma_2$ that have $l$ vertices.
Denote by $\tau_2(l)=|\mathcal{T}_l|$ the number of such triangulations and choose $K=K_2$ as in Corollary~\ref{NumberTriangulations} such that $\tau_2(l)$ is at most $K^l$. Let $p = \sqrt{c/n}$ for some $c < 1/K$.

By Euler's formula, every triangulation $T$ of the oriented surface $\Sigma_g$ of genus $g$ satisfies $f_2(T)=2(|V(T)|-2+2g) = 2(|V(T)|+2)$, if $g=2$, and $\Pr\big[T\subseteq X^2(n,p)\big] \leq n^{|V(T)|}p^{f_2(T)}$.
Hence,
\begin{align*}
\sum_{l=1}^{n}\!\sum_{T \in\mathcal{T}_l}\!\Pr\big[T\subseteq X^2(n,p)\big]
&\!\leq\! \sum_{l=1}^{n} \tau_2(l)\!\cdot\!n^lp^{2(l+2)}\\
&\!\leq\! \Big(\frac{c}{n}\Big)^{2}\sum_{l=1}^{n} (cK)^l
 \!=\! \Big(\frac{c}{n}\Big)^{2} \Big(\frac{1-(cK)^{n+1}}{1-cK} -1\Big),
\end{align*}
which clearly converges to zero as $n$ goes to infinity.
\end{proof}

\paragraph{Concluding Remarks.}
For the random $2$-complex $X^2(n,p)$, we have shown that the property of having the complete complex $K_t^2$ as a topological minor has \change{a coarse}{the} threshold $p=\Theta(1/\sqrt{n})$ for any $t\geq 10$. For dimensions $>2$ we could show an upper bound of $O(n^{-1/k})$ for the threshold.

The corresponding lower bound for higher dimensional complexes is open. It does not seem likely that an approach as simple as the one presented here will work in higher dimensions. An essential ingredient of our proof is that the number of triangulations of any fixed surface with a fixed number $l$ of vertices is simply exponential in $l$. The proof also depends on the fact that $f_2(T)=2(|V(T)|+2)$ for any triangulation $T$ of $\Sigma_2$, the surface of genus $2$.
In higher dimensions, it is not clear which manifold could play the role of $\Sigma_2$. Furthermore, bounds on the numbers of triangulations that are simply exponential are not to be expected: The number of $k$-spheres with $l$ vertices, e.g., is known to be at least $2^{\Omega(l^{\lfloor k/2 \rfloor})}$ for $k>3$ \cite{Kalai:1988} and $2^{\Omega(l^{5/4})}$ for $k=3$ \cite{PfeifleZiegler:2004}. 

It is very likely that, just as for graphs, the threshold for complete subdivision containment is actually a sharp threshold. For the upper bound an approach towards proving sharpness might be to combine the basic idea used here with more sophisticated arguments on the random graphs involved.

%
%
%
%

\subsubsection*{Acknowledgements} We would like to thank Matt Kahle and the referee for helpful comments.

\bibliographystyle{abbrv}
\bibliography{SubdivisionContainment}

\end{document}